\documentclass[12pt]{amsart}\usepackage{amssymb}

\newtheorem{mainthm}{Theorem}
\renewcommand{\themainthm}{\Alph{mainthm}}
\newtheorem{thm}{Theorem}[section]
\newtheorem{lem}[thm]{Lemma}
\newtheorem{conj}[thm]{Conjecture}
\newtheorem{cor}[thm]{Corollary}
\newtheorem{prop}[thm]{Proposition}
\theoremstyle{definition}
\newtheorem{defn}[thm]{Definition}
\newtheorem{rem}[thm]{Remark}
\newtheorem{exmp}[thm]{Example}

\newcommand{\lr}{\ell_R}
\newcommand{\lrz}{\ell_{R_0}}\newcommand{\lro}{\ell_{R_1}}
\newcommand{\lrt}{\ell_{R_2}}\newcommand{\ls}{\ell_S}
\newcommand\Z{\mathbb{Z}}\newcommand\R{\mathbb{R}}
\newcommand{\size}[1]{\ensuremath{\vert #1 \vert}}
\newcommand{\onto}{\twoheadrightarrow}

\title[Reflection length]{Bounding reflection length in\\ an affine
  Coxeter group}

\author[J. McCammond]{Jon McCammond}
\thanks{Work of McCammond partially supported by an NSF grant}
\author[T. K. Petersen]{T. Kyle Petersen}
\thanks{Work of Petersen partially supported by an NSA Young Investigator grant}
\date{\today}

\begin{document}

\begin{abstract}
  In any Coxeter group, the conjugates of elements in the standard
  minimal generating set are called reflections and the minimal number
  of reflections needed to factor a particular element is called its
  reflection length.  In this article we prove that the reflection
  length function on an affine Coxeter group has a uniform upper
  bound.  More precisely we prove that the reflection length function
  on an affine Coxeter group that naturally acts faithfully and
  cocompactly on $\R^n$ is bounded above by $2n$ and we also show that
  this bound is optimal.  Conjecturally, spherical and affine Coxeter
  groups are the only Coxeter groups with a uniform bound on
  reflection length.
\end{abstract}
\maketitle

Every Coxeter group $W$ has two natural generating sets: the finite
set $S$ used in its standard finite presentation and the set $R$ of
reflections formed by collecting all conjugates of the elements in
$S$.  The first generating set leads to the standard length function
$\ls\colon W \to \mathbb{N}$ and the second is used to define the
reflection length function $\lr\colon W\to \mathbb{N}$.  When $W$ is
finite, both length functions are uniformly bounded for trivial
reasons and are fairly well understood.\footnote{For finite $W$ these
  generating sets and length functions exhibit an interesting duality:
  the maximum value of $\ls$ is $\size{R}$ and the maximum value of
  $\lr$ is $\size{S}$.  See \cite{Bessis03} for further details and
  for additional illustrations of this phenomenon.}  For infinite
Coxeter groups the function $\ls$ is always unbounded because there
are only finitely many group elements of a given length as a
consequence of the fact that $S$ is finite.  Our main result is that
$\lr$ remains bounded for affine Coxeter groups and we provide an
explicit optimal upper bound.

\begin{mainthm}[Explicit affine upper bounds]\label{main:affine}
  If $W$ is an affine Coxeter group that naturally acts faithfully and
  cocompactly on $\R^n$ then every element of $W$ has reflection
  length at most $2n$ and there exist elements in $W$ with reflection
  length equal to $2n$.
\end{mainthm}

The article is structured as follows.  The first two sections recall
basic facts, the third establishes a key technical result and the
fourth contains the proof of our main result.  The final section
contains a conjecture about infinite non-affine Coxeter groups.

\section{Reflection length}\label{sec:length}

We assume the reader is familiar with the basic theory of reflection
groups (as described, for example, in \cite{Humphreys90}) and we
generally follow the standard notational conventions.

\begin{defn}[Reflection length]\label{def:lr}
  Let $W$ be a Coxeter group with standard generating set $S$. A
  \emph{reflection} in $W$ is any conjugate of an element of $S$ and
  we use $R$ to denote the set of all reflections in $W$.  In other
  words, $R=\{ wsw^{-1} \mid s \in S, w \in W\}$.  We should note that
  unless $W$ is a finite group, $R$ is an infinite set.  For any
  element $w \in W$, its \emph{reflection length} $\lr(w)$ is the
  minimal number of reflections whose product is $w$.  Thus $w = r_1
  r_2 \cdots r_k$ with $r_i \in R$ means $\lr(w) \leq k$.
  Alternatively, $\lr(w)$ can be defined as the combinatorial distance
  from the vertex labeled by the identity to the vertex labeled by $w$
  in the Cayley graph of $W$ with respect to $R$.
\end{defn}

Since combinatorial distance defines a metric on the vertex set of any
graph and Cayley graphs are homogeneous in the sense that there is a
vertex-transitive group action, metric properties of the distance
function translate into properties of $\lr$.  Symmetry and the
triangle inequality, for example, imply that $\lr(w) = \lr(w^{-1})$,
and $\lr(uv)\leq \lr(u) + \lr(v)$, respectively.  

It is sufficient to investigate reflection length in irreducible
Coxeter groups because of the following elementary result.

\begin{prop}[Reducible Coxeter groups]\label{prop:reducible}
  When $W$ is a reducible Coxeter group, its standard generating set
  $S$ has a nontrivial partition $S = S_1 \sqcup S_2$ in which every
  element in $S_1$ commutes with every element in $S_2$.  In this
  context, $W = W_1 \times W_2$ where $W_i$ denotes the parabolic
  subgroup generated by $S_i$, the reflections $R$ in $W$ can be
  partitioned as $R=R_1 \sqcup R_2$ where $R_i$ denotes the
  reflections in $W_i$ and when $w \in W$ is written in the form
  $w=w_1w_2$ with $w_i \in W_i$, we have $\lr(w) = \lro(w_1) +
  \lrt(w_2)$.
\end{prop}

\section{Affine Coxeter groups}\label{sec:affine}

Next we review the construction of an affine Coxeter group from a
crystallographic root system.

\begin{defn}[Affine Coxeter groups]\label{def:affine}
  Recall that a crystallographic root system $\Phi$ is a finite
  collection of vectors that span a real Euclidean space $V$
  satisfying a few elementary properties and that an affine Coxeter
  group $W$ can be constructed from $\Phi$ as follows.  For each
  $\alpha \in \Phi$ and $i\in \Z$ let $H_{\alpha,i}$ denote the
  (affine) \emph{hyperplane} of solutions to the equation $\langle
  x,\alpha \rangle = i$ where the brackets denote the standard inner
  product on $V$.  The unique nontrivial isometry of $V$ that fixes
  $H_{\alpha,i}$ pointwise is a \emph{reflection} that we call
  $r_{\alpha,i}$.  The collection $R = \{r_{\alpha,i} \mid \alpha \in
  \Phi, i \in \Z\}$ generates the affine Coxeter group $W$ and $R$ is
  its set of reflections in the sense of Definition~\ref{def:lr}.  A
  standard minimal generating set $S$ can be obtained by restricting
  to those reflections that reflect across the facets of a certain
  polytope in $V$.
\end{defn}

\begin{rem}[No finite factors]\label{rem:no-finite-factors}
  Every irreducible affine Coxeter group can be constructed from its
  crystallographic root system as described in
  Definition~\ref{def:affine} but the construction also works equally
  well when the root system is reducible.  It is not, however,
  sufficient to construct arbitrary reducible affine Coxeter groups
  because it always constructs affine Coxeter groups with affine
  irreducible components (i.e. no finite irreducible components).  The
  affine Coxeter groups constructible from a root system in this way
  can also be characterized as those that naturally act faithfully and
  cocompactly on some Euclidean space.
\end{rem}

For each affine Coxeter group $W$ constructed from a root system
$\Phi$ there is a finite Coxeter group $W_0$ related to $W$ in two
distinct ways.

\begin{defn}[Subgroups and quotients]\label{def:finite}
  We abbreviate $r_{\alpha,0}$ and $H_{\alpha,0}$ as $r_\alpha$ and
  $H_\alpha$, respectively.  The hyperplanes $H_\alpha$ are precisely
  the ones that contain the origin and the reflections $R_0 =
  \{r_\alpha \mid \alpha \in \Phi\}$ generate a finite Coxeter group
  $W_0$ that contains all elements of $W$ that fix the origin.  This
  embeds $W_0$ as a subgroup of $W$.  There is a group homomorphism
  $p\colon W \onto W_0$ defined by sending each generating reflection
  $r_{\alpha,i}$ in $W$ to $r_\alpha$ in $W_0$.
\end{defn}

Because the map $p\colon W \onto W_0$ sends $R$ to $R_0$, it sends
reflection factorizations to reflection factorizations, proving the
following.

\begin{prop}[Lengths and quotients]\label{prop:quotient-length}
  If the map $p:W \onto W_0$ sends $u$ to $u_0$ then $\lr(u) \geq
  \lrz(u_0)$.
\end{prop}
  
Finite Coxeter groups such as $W_0$ are also known as spherical
Coxeter groups because they act by isometries on spheres and we
digress for a moment to recall a few of their key properties.  In a
finite Coxeter group, reflection length has a geometric interpretation
that yields a spherical version of Theorem~\ref{main:affine} as an
immediate corollary.  The following properties were observed by Carter
in \cite{Carter72}.

\begin{prop}[Spherical reflection length]\label{prop:spherical}
  The reflection length of an element $w$ in a finite Coxeter group
  $W_0$ is equal to the codimension of the subspace of vectors that
  $w$ fixes in the standard orthogonal representation of $W_0$.  In
  addition, a reflection factorization $w = r_1r_2\cdots r_m$ is of
  minimum length if and only if the vectors normal to the hyperplanes
  of these reflections are linearly independent.
\end{prop}

\begin{cor}[Spherical upper bounds]\label{cor:spherical}
  Let $W_0$ be a finite Coxeter group whose standard representation
  acts on $\R^n$ by orthogonal transformations. For all $w \in W_0$,
  $\lr(w) \leq n$ and for elements that only fix the origin, $\lr(w) =
  n$.  More concretely, multiplying the elements in its standard
  minimal generating set $S_0$ produces an element $w$ with $\lr(w) =
  n$.
\end{cor}

Proposition~\ref{prop:spherical} also has consequences for some
elements in $W$.

\begin{cor}[Linearly independent roots]\label{cor:lin-ind}
  If $w = r_1 r_2 \cdots r_m$ is a reflection factorization of $w\in
  W$ in which the roots of the reflections are linearly independent,
  then $\lr(w) = m$.
\end{cor}

\begin{proof}
  The factorization shows $\lr(w) \leq m$.  For the lower bound note
  that each reflection $r_i \in R$ is $r_{\alpha_i,c_i}$ for some
  $\alpha_i \in \Phi$ and $c_i \in \Z$, and by hypothesis the roots
  $\alpha_i$ are linearly independent.  If we let $w_0 = p(w)$ then by
  Propositions~\ref{prop:quotient-length} and~\ref{prop:spherical} we
  have $\lr(w) \geq \lrz(w_0) = m$.
\end{proof}

Although we never use the following result established by Solomon in
\cite{Solomon63}, we mention it because it highlights how reflection
length captures fundamental aspects of the behavior of finite Coxeter
groups.

\begin{prop}[Solomon's factorization formula]\label{prop:solomon}
  For each finite Coxeter group $W_0$, the polynomial recording the
  distribution of reflection lengths factors completely over the
  integers.  In particular, \[f(x) = \sum_{w \in W_0} x^{\lr(w)} =
  \prod_{i=1}^n (1+ e_ix),\] where the $e_i$ are the exponents of
  $W_0$.
\end{prop}

We now return to the structure of the affine Coxeter group $W$
constructed from a root system $\Phi$.

\begin{defn}[Coroots]
  For each root $\alpha \in \Phi$, there is a corresponding
  \emph{coroot} $\alpha^\vee = c \alpha$ with $c= \frac{2}{\langle
    \alpha,\alpha \rangle}$.  The collection of all coroots is denoted
  $\Phi^\vee$ and the integral linear combinations of vectors in
  $\Phi^\vee$ is a lattice $L = \Z \Phi^\vee \cong \Z^n$ called the
  \emph{coroot lattice}.  The isomorphism with $\Z^n$ is a result of
  the existence of a subset $\Delta^\vee \subset \Phi^\vee$ of
  linearly independent vectors that form a basis for $L$.
\end{defn}

The coroot lattice describes the translations in $W$.

\begin{defn}[Translations]
  A \emph{translation} is a map that shifts each point by the same
  vector $\lambda$ and we let $t_\lambda$ denote the map sending each
  point $x \in V$ to $x + \lambda$.  For each $\alpha \in \Phi$
  consider the product $r_{\alpha,1} r_\alpha$.  Reflecting through
  parallel hyperplanes produces a translation in the $\alpha$
  direction and the exact translation is $t_{\alpha^\vee}$.  Because
  $t_\mu t_\nu = t_{\mu + \nu}$, there is a translation of the form
  $t_\lambda$ in $W$ for each $\lambda \in L$ and the set $T = \{
  t_\lambda \mid \lambda \in L\}$ forms an abelian subgroup of $W$. In
  fact, these are the only translations that are contained in $W$.
\end{defn}

The subgroup $T$ is also the kernel of the map $p\colon W \onto W_0$
and $W$ can be viewed as a semidirect product $W =W_0 \ltimes T$.

\begin{prop}[Normal forms]\label{prop:normal-form}
  For each element $w\in W$ there is a unique factorization $w =
  t_\lambda w_0$ where $t_\lambda$ is a translation with $\lambda\in
  L$ and $w_0$ is an element in $W_0$.
\end{prop}

\begin{proof}
  If such an expression exists then $w_0$ is the image of $w$ under
  the map $p\colon W \onto W_0$ and $\lambda$ is the image of the
  origin under $w$ (keeping in mind that elements of $W$ act on $V$ in
  function notation so that composition is from right to left).  This
  proves uniqueness.  For existence define $w_0 = p(w) \in W_0 \subset
  W$ and consider the element $ww_0^{-1}$.  Since it is in the kernel
  of $p$, it is translation in the form $t_\lambda$ for some
  $\lambda\in L$.
\end{proof}

\section{Translation dimension}

In this section we introduce the notion of the dimension of a vector
in the coroot lattice.  As above, $W$ is an affine Coxeter group,
constructed from a root system $\Phi$, acting on a vector space $V$.

\begin{defn}[Real dimension]\label{def:real-dimension}
  We call a subspace of $V$ spanned by a collection of coroots in
  $\Phi^\vee$ a \emph{real coroot subspace} and we say a vector
  $\lambda$ has \emph{real dimension} $k$ when $\lambda$ is contained
  in a $k$-dimensional coroot subspace but it is not contained in any
  real coroot subspace of strictly smaller dimension.  Since it is
  always possible find a coroot basis for each coroot subspace, real
  dimension $k$ means that $\lambda = c_1 \alpha_1^\vee + c_2
  \alpha_2^\vee + \cdots + c_k \alpha_k^\vee$ for some $c_i \in \R$
  and $\alpha_i^\vee \in \Phi^\vee$ but that no such expression exists
  with fewer summands.
\end{defn}

\begin{rem}[Non-unique subspaces]\label{rem:non-unique}
  There may be more than one $k$-dimensional coroot subspace
  containing a vector $\lambda \in L$ of real dimension $k$ because
  coroot subspaces are not closed under intersection.  Consider the
  $D_n$ root system $\Phi = \{\pm e_i \pm e_j \mid 1 \leq i < j \leq
  n\}$.  Every root has norm $2$ so $\alpha = \alpha^\vee$ and $\Phi =
  \Phi^\vee$.  The translation $t_\lambda$ with $\lambda = 2e_1 \in L$
  is $2$-dimensional because $2e_1$ is in the span of $\{e_1 + e_j,
  e_1 - e_j\}$ for any $j>1$ but as $j$ varies we get distinct
  $2$-planes containing $\lambda$.  They intersect along the
  $e_1$-axis but this line is not a real coroot subspace.
\end{rem}

The following definition is very similar.

\begin{defn}[Integral dimension]\label{def:dimension}
  We say a vector $\lambda \in L$ has \emph{integral dimension} $k$
  when $\lambda$ can be expressed in the form $\lambda = c_1
  \alpha_1^\vee + c_2 \alpha_2^\vee + \cdots + c_k \alpha_k^\vee$ for
  some $c_i \in \Z$ and $\alpha_i^\vee \in \Phi^\vee$ but no such
  expression exists with fewer summands.  The restriction to integral
  coefficients means that only vectors in $L$ have a dimension in this
  sense.  When $\lambda$ has integral dimension $k$ we say $\lambda$
  is \emph{$k$-dimensional} as is the corresponding translation
  $t_\lambda$. Note that the real dimension of $\lambda$ is a lower
  bound on its integral dimension.
\end{defn}

\begin{prop}[Dimensions]\label{prop:dimension}
  If $W$ is an affine Coxeter group constructed from a root system
  $\Phi$ and naturally acts faithfully and cocompactly on $V = \R^n$,
  then every vector in its coroot lattice has integral dimension at
  most $n$ and $n$-dimensional vectors do exist.
\end{prop}

\begin{proof}
  The first assertion is a consequence of the fact that $L \cong \Z^n$
  is a lattice with a $\Z$-basis in $\Phi^\vee$.  For the second assertion
  note that any vector $\lambda \in L$ that does not lie in the union
  of the finite number of proper subspaces through the origin that are
  spanned by coroots has real dimension $n$ and thus integral
  dimension at least $n$.
\end{proof}

The integral dimension of a vector $\lambda \in L$ bounds how hard it
is to move the origin to $\lambda$ using reflections.  To prove this
assertion we need an elementary result about factorizations.

\begin{lem}[Rewriting factorizations]\label{lem:rewriting}
  Let $W$ be a Coxeter group with reflections $R$ and let $w = r_1 r_2
  \cdots r_m$ be a reflection factorization. For any selection $1 \leq
  i_1 < i_2 < \cdots < i_k \leq m$ of positions there is a length~$m$
  reflection factorization of $w$ whose first $k$ reflections are
  $r_{i_1} r_{i_2} \cdots r_{i_k}$ and another length~$m$ reflection
  factorization of $w$ where these are the last $k$ reflections in the
  factorization.
\end{lem}

\begin{proof}
  Because reflections are closed under conjugation, for any
  reflections $r$ and $r'$ there exist reflections $r''$ and $r'''$
  such that $r r' = r'' r$ and $r' r = r r'''$.  Iterating these
  rewriting operations allows us to move the selected reflections into
  the desired positions without altering the length of the
  factorization.
\end{proof}

In preparation for the next result recall that $t_{\alpha^\vee} =
r_{\alpha,1}r_\alpha$ and note that because the hyperplanes
$H_{\alpha,i}$ are evenly spaced, the product $r_{\alpha,i+1}
r_{\alpha,i}$ is also $t_{\alpha^\vee}$ for every $i\in \Z$. More
generally $r_{\alpha,i+j} r_{\alpha,i} = t_{j \alpha^\vee}$.

\begin{prop}[Moving points]\label{prop:moving-pts}
  If $\lambda \in L$ has integral dimension~$k$ then $\lr(t_\lambda)
  \leq 2k$ and there is an element $u \in W$ with $\lr(u) \leq k$ that
  sends the origin to $\lambda$.
\end{prop}

\begin{proof}
  By the definition of integral dimension there is an equation of the
  form $\lambda = c_1 \alpha_1^\vee + c_2 \alpha_2^\vee + \cdots + c_k
  \alpha_k^\vee$ and the formulas $t_{\mu + \nu} = t_\mu t_\nu$ and
  $t_{j\alpha^\vee} = r_{\alpha,j} r_\alpha$ show that $t_\lambda$ has
  a length $2k$ reflection factorization of the form $t_\lambda =
  (r_{\alpha_1,c_1} r_{\alpha_1}) (r_{\alpha_2,c_2}
  r_{\alpha_2})\cdots (r_{\alpha_k,c_k} r_{\alpha_k})$.  Next, by
  Lemma~\ref{lem:rewriting} there is another length $2k$ reflection
  factorization of $t_\lambda$ where the final $k$ reflections are
  $r_{\alpha_1} r_{\alpha_2} \cdots r_{\alpha_k}$.  Since all of these
  reflections fix the origin, the product of the first $k$ reflections
  is an element $u$ that sends the origin to $\lambda$.
\end{proof}

Note that just as there can be distinct minimal expressions for
$\lambda$ (Remark~\ref{rem:non-unique}), there are often distinct
elements of reflection length $k$ that send the origin to $\lambda$.

\begin{thm}[Equivalent definitions]\label{thm:equiv}
  For each $\lambda \in L$ the real dimension of $\lambda$, the
  integral dimension of $\lambda$ and the minimal reflection length of
  an element sending the origin to $\lambda$ are equal.
\end{thm}

\begin{proof}
  Let $k_r$, $k_d$ and $k_m$ be three numbers at issue in the order
  listed.  Proposition~\ref{prop:moving-pts} shows that $k_d \geq
  k_m$.  Next, let $u$ be an element sending the origin to $\lambda$.
  Because a reflection $r_{\alpha,i}$ move points in the $\alpha^\vee$
  direction, $\lambda$ is in the span of the coroots associated to the
  reflections in a minimal length reflection factorization of $u$.
  Thus $k_m \geq k_r$.  And finally, let $V'$ be a $k_r$-dimensional
  coroot subspace of $V$ containing $\lambda$.  The set $\Phi' = \Phi
  \cap V'$ satisfies the requirements to be a root system and $L' = L
  \cap V'$ is its coroot lattice.  Because $\lambda$ lies in the
  lattice $L'$, $k_r \geq k_d$.  The combination $k_d \geq k_m \geq
  k_r \geq k_d$ shows all three are equal.
\end{proof}

\section{Bounding reflection length}

We are now ready to prove our main result.

\begin{prop}[Bounds]\label{prop:bound}
  If $W$ is a affine Coxeter group and $w \in W$ has the form $w =
  t_\lambda w_0$ where $t_\lambda$ is a $k$-dimensional translation
  and $w_0 \in W_0$ is an element fixing the origin then $k\leq \lr(w)
  \leq k+n$.  In particular, every element $w \in W$ has $\lr(w) \leq
  2n$.
\end{prop}

\begin{proof}
  The lower bound, $\lr(w) \geq k$, follows from Theorem
  \ref{thm:equiv}.  By Proposition~\ref{prop:moving-pts} and Theorem
  \ref{thm:equiv} there is an element $u$ with $\lr(u) = k$ that sends
  the origin to $\lambda$.  Because the element $v_0 = u^{-1}w$ fixes
  the origin, it is in $W_0$ and $\lr(v_0) \leq n$ by
  Corollary~\ref{cor:spherical}.  Thus $\lr(w) = \lr(u v_0) \leq
  \lr(u) + \lr(v_0) \leq k+n$.  For the final assertion note that
  every element can be written in this form
  (Proposition~\ref{prop:normal-form}) and every translation is
  $k$-dimensional for some $k\leq n$
  (Proposition~\ref{prop:dimension}).
\end{proof}

Based on the proof of this proposition, one might conjecture that each
$w = t_\lambda w_0$ with $k$-dimensional $\lambda$ has another
factorization $w = uv_0$ with $v_0 \in W_0$ and $\lr(w) = \lr(u) +
\lr(v_0)$. This is not always the case.

\begin{exmp}[Exact bounds]
  Let $\alpha_{ij}$ denote the vector $e_i - e_j$, let $\Phi = \{
  \alpha_{ij} \mid 1 \leq i,j \leq 4\}$ be the $A_3$ root system with
  coroot lattice $L$, and consider the elements $w_{ijk} =
  r_{\alpha_{12},i}\ r_{\alpha_{23},j}\ r_{\alpha_{34},k}$.  By
  Corollary~\ref{cor:lin-ind} $\lr(w_{ijk}) = 3$.  All of these
  elements are sent under $p\colon W \onto W_0$ to $w_{000} =
  r_{\alpha_{12}} r_{\alpha_{23}}r_{\alpha_{34}}$ which cyclically
  permutes the four coordinates moving the point $(x,y,z,w)$ to
  $(w,x,y,z)$.  The minimum length reflection factorizations of
  $w_{000}$ in $W_0$ are well-known and they are encoded by maximal
  chains in the lattice $NC_4$ of non-crossing partitions on four
  elements \cite{McCammond06}.  Every reflection in $R_0$ occurs in some
  factorization of $w_{000}$ but there exists a pair of reflections,
  $r_{\alpha_{13}}$ and $r_{\alpha_{24}}$, representing a ``crossing''
  partition that cannot both occur in the same factorization.  By
  varying $i$, $j$ and $k$, the elements $w_{ijk}$ produce all
  elements of the form $t_\lambda w_{000}$ with $\lambda \in L$. Thus
  for a careful choice of values, the vector $\lambda$ is in the
  span of $\alpha_{13}$ and $\alpha_{24}$ and this is the
  unique minimal dimensional coroot subspace containing $\lambda$.  If
  $u$ is a product of two reflections sending such a $\lambda$ to the
  origin, then the coroots involved in this product must be
  $\alpha_{13}$ and $\alpha_{24}$.  As a consequence, there is no
  length~$3$ reflection factorization of the form $w_{ijk} = uv_0$
  since the projection of such a factorization to $W_0$ would be a
  length~$3$ reflection factorization of $w_{000}$ containing both
  $r_{\alpha_{13}}$ and $r_{\alpha_{24}}$ which is known not to exist.
\end{exmp}

\begin{prop}[Optimality]\label{prop:optimal}
  If $W$ is a affine Coxeter group and $w = t_\lambda$ is a
  $k$-dimensional translation then $\lr(w) = 2k$.  In particular,
  $n$-dimensional translations have reflection length $2n$.
\end{prop}

\begin{proof}
  Being a translation in $W$, $w = t_\lambda$ for some $\lambda \in L$
  and by definition of dimension $\lambda = c_1 \alpha_1^\vee + c_2
  \alpha_2^\vee + \cdots + c_k \alpha_k^\vee$ for $\alpha_i^\vee \in
  \Phi^\vee$ and $c_i \in \Z$.  By Proposition~\ref{prop:moving-pts}
  $\lr(w) \leq 2k$.  To show that $2k$ is also a lower bound suppose
  that $w = r_1 r_2 \cdots r_m$ with each $r_i = r_{\alpha_i,c_i} \in
  R$.  Because the reflection $r_i$ moves points in the
  $\alpha_i^\vee$ direction and $\lambda$ has real dimension $k$
  (Theorem~\ref{thm:equiv}), the coroots $\alpha_i^\vee$ must span a
  subspace of dimension at least $k$.  Next, we may assume that the
  first $k$ reflections have linearly independent coroots since it is
  possible to use Lemma~\ref{lem:rewriting} to move any $k$
  reflections with linearly independent coroots to the front.  This
  alters the reflection factorization of $w$ but leaves the total
  length unchanged.  Let $u = r_1r_2 \cdots r_k$ and let $v = r_{k+1}
  \cdots r_m$ so that $w=uv$, let $u_0 = p(u)$ and $v_0 = p(v)$ where
  $p$ is the homomorphism $p\colon W \onto W_0$ and note that
  $\lrz(u_0) = k$ by Proposition~\ref{prop:spherical}.  Since $w$ is a
  translation, $p(w)$ is the identity and $v_0 = u_0^{-1}$.  Finally,
  by Proposition~\ref{prop:quotient-length} we have $\lr(v) \geq
  \lrz(v_0)= \lrz(u_0^{-1}) = \lrz(u_0) = k$, which means that any
  reflection factorization of $v$ has length at least $k$.  This
  implies that $m \geq 2k$ and as a consequence every reflection
  factorization of $w$ has length at least $2k$.
\end{proof}

\renewcommand{\themainthm}{\ref{main:affine}}
\begin{mainthm}[Explicit affine upper bounds]
  If $W$ is an affine Coxeter group that naturally acts faithfully and
  cocompactly on $\R^n$ then every element of $W$ has reflection
  length at most $2n$ and there exist elements in $W$ with reflection
  length equal to $2n$.
\end{mainthm}

\begin{proof}
  Proposition~\ref{prop:bound} shows that $2n$ is an upper bound.
  Propositions~\ref{prop:dimension} and~\ref{prop:optimal} show that
  it is optimal.
\end{proof}

\begin{rem}[Finite factors]
  The cocompactness assumption in Theorem~\ref{main:affine}
  essentially means that $W$ does not have any finite irreducible
  factors (Remark~\ref{rem:no-finite-factors}).  When finite
  irreducible factors are present, the optimal upper bound can be
  lowered accordingly.  In particular, if $W = W_f \times W_a$ where
  $W_f$ is the product of the finite irreducible factors, $W_a$ is the
  product of the affine irreducible factors and $\R^n = \R^{n_f}
  \oplus \R^{n_a}$ is the orthogonal decomposition preserved by this
  splitting then by Proposition~\ref{prop:reducible},
  Corollary~\ref{cor:spherical} and Theorem~\ref{main:affine} the
  optimal upper bound for the reflection length function on $W$ is
  $n_f + 2n_a = 2n-n_f$.
\end{rem}

We should note that we have not found an elementary way to compute
$\lr(w)$ exactly for a generic $w$ in an affine Coxeter group $W$.
Techiques for easily computing $\lr(w)$ would be useful.  For example,
they would enable one to investigate whether the polynomial
$f_\lambda(x) =\sum_{w\in W_0} x^{\lr(t_\lambda w)}$ for $\lambda \in
L$ has properties similar to $f_0(x)$, the Solomon polynomial
discussed in Proposition~\ref{prop:solomon}.  All we can say at the
moment (as a consequence of Proposition \ref{prop:bound}) is that when
$\lambda$ is $k$-dimensional, the polynomial $f_{\lambda}(x)$ has
degree at most $k+n$ and is divisible by $t^k$.

\section{An unbounded example}\label{sec:unbounded}

When $W$ is a Coxeter group that is neither finite nor affine we
conjecture the following.

\begin{conj}[No upper bounds]\label{conj:no-bound}
  The reflection length function on a Coxeter group $W$ has a uniform
  upper bound if and only if $W$ is of spherical or affine type.
\end{conj}

We are not currently able to prove this conjecture but we can show
that it holds in at least one special case.  Let $W$ be the free
Coxeter group on three generators, i.e. the group generated by three
involutions and with no other relations.  Using a criterion of Dyer
\cite{Dyer01} we can show that the $n$-th power of the product of the
three standard generators has reflection length $n+2$.  In particular,
the reflection length function is unbounded on $W$ and
Conjecture~\ref{conj:no-bound} holds in this case.  We note that $W$
can be viewed as a reflection group acting on the hyperbolic plane
generated by the reflections in the sides of an ideal triangle.
Proving that all hyperbolic triangle Coxeter groups have unbounded
reflection length would be first step towards proving
Conjecture~\ref{conj:no-bound}.

\medskip
\noindent {\bf Acknowledgements} We would like to thank Matthew Dyer
and John Stembridge for helpful early conversations about the ideas in
this note and Rob Sulway for conversations about related topics that
arise in his dissertation.


\end{document}